\documentclass[12pt,reqno]{amsart}
\usepackage{amsmath, amsthm, amssymb, stmaryrd}
\usepackage{hyperref}
\usepackage{enumerate}
\usepackage{enumitem}
\usepackage{url}
\usepackage{tikz-cd}

\let\OLDthebibliography\thebibliography
\renewcommand\thebibliography[1]{
  \OLDthebibliography{#1}
  \setlength{\parskip}{0pt}
  \setlength{\itemsep}{0pt plus 0.3ex}
}

\usepackage{mathtools}

\usepackage{multirow, array}
\usepackage{placeins}

\usepackage{caption} 
\advance \topmargin by -\headheight
\advance \topmargin by -\headsep
     
\evensidemargin \oddsidemargin
\newtheorem{thm}{Theorem}[section]
\newtheorem{lemma}[thm]{Lemma}
\newtheorem{prop}[thm]{Proposition}

\theoremstyle{definition}
\newtheorem{defn}[thm]{Definition}
\newtheorem{example}[thm]{Example}
\newtheorem{obs}[thm]{Observation}

\theoremstyle{remark}

\numberwithin{equation}{section}

\newcommand*\wrapletters[1]{\wr@pletters#1\@nil}
\def\wr@pletters#1#2\@nil{#1\allowbreak\if&#2&\else\wr@pletters#2\@nil\fi}

\def \bbA {\mathbb A}

\def \bC {\mathbb C}

\def \bP {\mathbb P}
\def \bQ {\mathbb Q}
\def \bR {\mathbb R}
\def \bZ {\mathbb Z}

\def \deg {\mathrm{deg}}

\def \Res {{\mathrm{Res}}}
\def \Gal {{\mathrm{Gal}}}
\def \Aut{{\mathrm{Aut}}}

\begin{document}
\title[Lines of polynomials with alternating Galois group]{Lines of polynomials with alternating Galois group}
\author[Nuno Arala]{Nuno Arala}
\email{Nuno.Arala-Santos@warwick.ac.uk}
\thanks{}
\date{}
\begin{abstract} We study the problem of existence of one-parameter, linear families of polynomials of degree $n$ all of whose polynomials have Galois group $A_n$. The methods we use have a strong geometric flavour.
\end{abstract}
\maketitle

\section{Introduction}
\label{intro}

For a positive integer $n\geq3$ and a transitive subgroup $G\leq S_n$, denote by $N_G(H)$ the number of polynomials
$$f(X)=X^n+a_{n-1}X^{n-1}+\cdots+a_0$$
whose coefficients $a_{n-1},\ldots,a_0$ are integers lying in the interval $[-H,H]$, and such that the Galois group $G_f$ of $f$ over $\bQ$ is isomorphic to $G$. Recall that the Galois group of $f(X)$ is the Galois group of the splitting field of $f$ over $\mathbb{Q}$, which we can regard as a subgroup of $S_n$ via its action on the roots of $f$. Starting in the 19th century with Hilbert and his so-called Irreducibility Theorem, which can be thought of as the statement that $N_{G}(H)=o(H^n)$ whenever $G\neq S_n$, there has been a lot of interest in establishing upper bounds for the numbers $N_G(H)$ for different groups $G$. Recent progress on this topic includes the work of Chow and Dietmann \cite{CD2020} on the cases $n=3,4$, the work of Xiao \cite{Xiao} on $A_3$-cubics and Bhargava's paper \cite{Bhargava} proving the long-standing Van der Waerden's conjecture.

The work we present in this paper arose from an attempt to complement this work with a focus on the dual goal of constructing large families of polynomials with prescribed Galois groups. There is an abundance of work on constructive approaches to the Inverse Galois Problem in the literature, as an example of which we refer the reader to Smith's work \cite{Smith}. We focus on one-parameter families of polynomials of degree $n$ with Galois group $G=A_n$, the alternating group consisting of the even permutations in $S_n$. If the coefficients of the polynomials in our family are themselves polynomials of degree at most $k$ in the defining parameter $T$, then the number of polynomials in the family with coefficients bounded by $H$ is $\gg H^{\frac{1}{k}}$. Since the goal is to construct large families it is natural to give special attention to those for which we can take $k=1$, i.e. those that depend linearly on a parameter.

This leads to the study of \emph{lines} of polynomials with Galois group $A_n$. By this we mean a linear family of polynomials parameterized by an integer $T$, of the form
\begin{equation}
P_T(X)=P(X)-TQ(X)
\label{lines}
\end{equation}
where $P\in\mathbb{Z}[X]$ is a monic polynomial of degree $n$ and $Q\in\mathbb{Z}[X]$ is a nonzero polynomial of degree at most $n-1$, such that $G_{P_T}\cong A_n$ for every integer $T$. Two questions arise naturally:
\begin{enumerate}[label=(\arabic*)]
\item For which values of $n\geq 3$ does there exist a line of polynomials as described in \eqref{lines}?
\item For those values of $n$, can we give a complete classification of such lines?
\end{enumerate}
The present paper is devoted to some progress toward these questions. Following ideas from Mestre \cite{Mes1990}, we are able to prove that such lines exist for every odd $n$.
\begin{thm}
\label{oddcase}
Let $n\geq 3$ be odd. Then there exists a line of polynomials of degree $n$ with Galois group $A_n$ over $\bQ$.
\end{thm}
For even $n$ the situation is more complicated. Nevertheless we have been able to show that lines of polynomials of degree $n$ with alternating Galois group exist for infinitely many even values of $n$. This is the content of the following Theorem.
\begin{thm}
\label{n-1square}
Let $n$ be an even positive integer such that $n-1$ is a square. Then there exists a line of polynomials of degree $n$ with Galois group $A_n$ over $\bQ$.
\end{thm}
In order to state our next results it is convenient to explain a natural extension of the main problem. It is natural to wonder what happens if one relaxes the condition that the polynomials in our lines must have Galois group $A_n$ over $\bQ$, but instead requires that they have Galois group $A_n$ over an appropriate number field. It turns out that, by just allowing quadratic extensions of $\bQ$, we obtain satisfactory lines for \emph{every} integer $n\geq3$. This is the content of the following result, which is proved in exactly the same way as Theorem \ref{n-1square}.
\begin{thm}
\label{evencase}
Let $n\geq3$ be even. Then there exists a line of polynomials of degree $n$ with Galois group $A_n$ over $K$, where $K$ is some number field with $[K:\bQ]\leq2$.
\end{thm}
We then do an exhaustive analysis of the quartic case. We are actually able to classify all the lines of polynomials of degree $4$ with Galois group $A_4$ over $\bQ$, and this is the content of the next result we present. However, in order to make the presentation more compact, it is convenient to change our language a bit. First of all we will no longer impose that the coefficients of the polynomials defining our lines are integers, but only rationals. It is straightforward to go from a line of polynomials with rational coefficients to a line of polynomials with integer coefficients if one wants to do so. The natural alternating condition then becomes the statement that the polynomial $P-TQ$, viewed as a polynomial in $\bQ(T)[X]$, has Galois group $A_n$ over $\bQ(T)$. In other words, we look for lines whose polynomials are \emph{generically} $A_n$-polynomials. Hilbert's Irreducibility Theorem guarantees that this condition and our original one are not too far apart, and indeed we will observe later (Observation \ref{obs}) that it is easy to go from a line of $A_n$-polynomials in this apparently weaker sense to a line of $A_n$-polynomials in our original sense, so that this change is more for notational convenience. We introduce also a useful equivalence relation on lines of polynomials. Given a field $K$, we may define a right action of the group $\mathrm{GL}_2(K)$ on the space of polynomials in $K[X]$ of degree at most $n$ as follows: if
$$\gamma=\begin{pmatrix}a&b\\c&d\end{pmatrix}$$
then
$$(P\gamma)(X)=P\left(\frac{aX+b}{cX+d}\right)(cX+d)^n\text{.}$$
We observe that this action preserves lines of polynomials. Moreover, given a line $P-TQ$ of polynomials of degree $n$, we may define a left action on it by mapping it under $\gamma$ to the new line $P^\star-TQ^\star$ where $P^\star=aP+bQ$ and $Q^\star=cP+dQ$. These two actions commute with each other, and both preserve the generic Galois group (i.e. the Galois group over $K(T)$). We will say that two lines of polynomials in $K[X]$ are \emph{equivalent} if they lie in the same orbit with respect to this double action of $\mathrm{GL}_2(K)$. We are now ready to state our classification result.
\begin{thm}
\label{class4}
Any line of polynomials of degree $4$ in $\bQ[X]$ with generic Galois group $A_4$ over $\bQ$ is equivalent to a line of the form \eqref{lines} with
$$P(X)=X^4-2kX^2-8mX+k^2$$
and
$$Q(X)=X^3+kX+m$$
where $k,m\in\bQ$ are such that $-4k^3-27m^2$ is a square in $\bQ$.
\end{thm}
Using unique factorization in the ring of integers of $\bQ(\sqrt{-3})$ it is not hard to parameterize the pairs $(k,m)$ of rationals for which $-4k^3-27m^2$ is a square, and doing so one obtains a more explicit form of Theorem \ref{class4}.

The language we introduced is necessary in order to fully appreciate the most subtle details of the quartic case. It turns out that, in the cases covered by Theorems \ref{oddcase} and \ref{n-1square}, the lines can be obtained in such a way that the corresponding Galois group of $P(X)-TQ(X)$ is $A_n$ over $\overline{\bQ}(T)$, and not just over $\bQ(T)$, where the bar denotes algebraic closure. In other words, given any number field $K$, a generic polynomial in such lines has Galois group $A_n$ over $K$, and not only over $\bQ$. This is in stark contrast with the quartic case. Even though there exist lines of quartic polynomials with Galois group $A_4$ over $\bQ$, all polynomials in any such line have Galois group strictly smaller than $A_4$ over some specific cubic extension of $\bQ$ that depends on the given line. In the notation of Theorem \ref{class4} this cubic extension is the splitting field of $X^3+kX+m$ over $\bQ$.

For a number field $K$, we will call lines that have Galois group $A_n$ both over $K(T)$ and over $\overline{\bQ}(T)$ \emph{strong alternating lines over $K$}, and those that have Galois group $A_n$ over $K(T)$ but strictly smaller Galois group over $\overline{\bQ}(T)$ \emph{weak alternating lines over $K$}. Our previous remark about the quartic case can be stated as follows.
\begin{thm}
\label{nostrong}
There are no strong alternating lines of quartic polynomials over $\bQ$.
\end{thm}
A natural follow-up question is to classify those number fields for which there exists a strong alternating line of quartic polynomials. The following result answers this question for quadratic number fields.
\begin{thm}
\label{hm}
Let $m$ be a square-free integer. Then there exists a strong alternating line of quartic polynomials over $\bQ(\sqrt{m})$ if and only if $m<0$ and $m\not\equiv1\pmod{8}$.
\end{thm}

Let us say here a few words about the ideas used in this paper. The main idea is to adopt a geometric perspective, using the correspondence between curves and fields from Algebraic Geometry to view extensions of $\bQ(T)$ as maps from a curve to $\bP^1_\bQ$. This enables us to relate the Galois group of the given extension to the ramification properties of the corresponding map, using the principle that relates the Galois group of an extension of complex function fields to the monodromy of the corresponding cover of complex curves. This is stated in \cite{Harris}, for example, and yields Theorem \ref{cover}, which we use. The quartic case, discussed in \S\ref{quarticsec}, leads using this approach to some Diophantine problems which turn out to admit very elegant solutions because the geometry of the underlying algebraic varieties is surprisingly simple. In the special case of Theorem \ref{hm} the theory of quadratic forms over local and global fields plays a special role. For that we refer the reader to \cite{Lam}.
\subsection{Structure of the paper}
In \ref{prel} we review some algebraic and geometric tools which will play an important role in our approach. In \S\ref{odd} we give a self-contained and elementary proof of Theorem \ref{oddcase}, without appealing to highbrow results such as Hilbert's Irreducibility Theorem. In \S\ref{even} we prove our results on general even degrees, namely Theorems \ref{n-1square} and \ref{evencase}. In \S\ref{quarticsec} we present our work on lines of polynomials of degree $4$. 

\subsection{Acknowledgements}
The author was funded through the Engineering and Physical Sciences Research Council Doctoral Training Partnership at the University of Warwick. Special thanks are due to Sam Chow for suggesting this problem and for enlightening discussions about it, as well as to Damiano Testa for suggesting a key idea and to Simon Myerson for countless suggestions during the elaboration of this manuscript.

\section{Preliminaries}
\label{prel}
For the ease of the reader, we collect here some standard results which will be used throughout the paper.

\subsection{Resultants, discriminants and resolvents}
The \emph{resultant} of two polynomials $f$ and $g$ with coefficients in a field $K$ is a symmetric function on the roots of $f$ and $g$ which, once written in terms of the coefficients of $f$ and $g$, can be used to check whether $f$ and $g$ have a common root in the algebraic closure $\overline{K}$ by a simple calculation done inside $K$. Its interest to us stems from its relation with discriminants, a property we will make extensive use of throughout the paper.

\begin{defn}
\label{defres}
Let $K$ be a field and let
$$f(X)=a_mX^m+a_{m-1}X^{m-1}+\cdots+a_0\quad\text{and}\quad g(X)=b_nX^n+b_{n-1}X^{n-1}+\cdots+b_0$$
be polynomials in $K[X]$, with $a_m,b_n\neq0$. Suppose that $f$ and $g$ factor over $\overline{K}$ as
$$f(X)=a_m(X-\alpha_1)\cdots(X-\alpha_m)\quad\text{and}\quad g(X)=b_n(X-\beta_1)\cdots(X-\beta_n)\text{.}$$
The \emph{resultant} of $f$ and $g$, denoted by $\Res(f,g)$, is defined by
\begin{align}
\Res(f,g)&=a_m^nb_n^m\prod_{i=1}^m\prod_{j=1}^n(\alpha_i-\beta_j)\nonumber\\
&=a_m^n\prod_{i=1}^mg(\alpha_i)\label{res}\text{.}
\end{align}
\end{defn}
The following properties follow easily from the definition.
\begin{prop}
\label{propres}
For any polynomials $f,g,h\in K[X]$, we have:
\begin{enumerate}[label=(\roman*)]
\item $\Res(g,f)=(-1)^{mn}\Res(f,g)$;
\item $\Res(f,g+fh)=a_m^{\deg(g+fh)-\deg(g)}\Res(f,g)$ (where $f(X)=a_mX^m+\cdots$);
\item $\Res(f,gh)=\Res(f,g)\Res(f,h)$.
\end{enumerate}
\end{prop}
\begin{proof}
Immediate from the definition. For (ii) and (iii) it is preferable to work with expression \eqref{res}.
\end{proof}
As announced before, there is a fruitful connection between resultants and discriminants.
\begin{defn}
Let $f$ be a polynomial in $K[X]$. Suppose that $f$ factors over $\overline{K}$ as
$$f(X)=c(X-\alpha_1)\cdots(X-\alpha_n)\text{.}$$
We define the \emph{discriminant} of $f$, denoted $\Delta(f)$, by
$$\Delta(f)=c^{2n-2}\prod_{1\leq i<j\leq n}(\alpha_i-\alpha_j)^2\text{.}$$
\end{defn}
Being a symmetric polynomial in $\alpha_1,\ldots,\alpha_n$, the discriminant $\Delta(f)$ is an element of $K$, and it is of particular relevance to work on polynomials with alternating Galois group because of the following standard fact.
\begin{thm}
\label{sqdisc}
Let $f\in K[X]$ be a separable polynomial of degree $n$, where $K$ is a field with $\mathrm{char}(K)\neq2$, and identify the Galois group of $f$ with a subgroup of $S_n$ in the natural way. Then the Galois group of $f$ is a subgroup of $A_n$ if and only if $\Delta(f)$ is a square in $K$.
\end{thm}
\begin{proof}
See \cite{Cox2012}, Theorem 7.4.1.
\end{proof}

Observe that the discriminant of a polynomial can be used to detect multiple roots. In fact, it is clear from the definition that $f$ has a multiple root if and only if $\Delta(f)=0$. On the other hand, it is also true that $f$ has a multiple root if and only if it has a common root with its formal derivative $f'$. Since two polynomials have a common root if and only if their resultant vanishes, it becomes natural to expect a relation between $\Delta(f)$ and $\Res(f,f')$. This relation, indeed, exists:
\begin{prop}
\label{resdisc}
For any nonzero polynomial $f\in K[X]$ with leading coefficient $c$ and degree $n$,
$$\Delta(f)=\frac{(-1)^{\frac{n(n-1)}{2}}}{c}\Res(f,f')\text{.}$$
\end{prop}
\begin{proof}
See \cite{Lan2002}, Proposition IV.8.5.
\end{proof}

We finally record some facts about the cubic resolvent, which we will need later.
\begin{defn}
Given a quartic polynomial
$f=X^4-c_1X^3+c_2X^2-c_3X+c_4$
in $K[X]$, where $K$ is a field of characteristic different from $2$, we define its \emph{cubic resolvent} to be the polynomial
$$X^3-c_2X^2+(c_1c_3-4c_4)X+c_3^2-c_1^2c_4+4c_2c_4\text{.}$$
\end{defn}
The main properties of the cubic resolvent that we will need are listed in the following proposition.
\begin{prop}
\label{cubicr}
Let $f\in K[X]$ be an irreducible quartic polynomial. Then,
\begin{itemize}
\item The cubic resolvent of $f$ has the same discriminant as $f$;
\item The Galois group of $f$ is either $S_4$ or $A_4$ if and only if its cubic resolvent is irreducible;
\item The Galois group of $f$ is isomorphic to the Klein four-group $V_4$ if and only if its cubic resolvent splits completely over $K$.
\end{itemize}
\end{prop}
\begin{proof}
See \cite{Cox2012}, Theorem 13.1.1. Note that in this book the term ``Ferrari resolvent'' is used instead of ``cubic resolvent''.
\end{proof}

\subsection{Galois theory of morphisms of curves}
Let $k$ denote an algebraically closed field of characteristic $0$. Suppose $X/k$ is a smooth, proper algebraic curve over $k$, and $f:X\to \bP^1_k$ is a morphism of curves of degree $d>0$. From the theory of curves this corresponds to a field extension $k(X)\supseteq k(t)$ of the corresponding function fields. It will be important for us to relate the Galois group of this field extension to the geometry of the map $f$. The main statement here is the following:

\begin{thm}
\label{cover}
Fix a total order on the set of points in $\bP^1_k$. We can assign to each point $y\in\bP^1_k$ a permutation $\sigma_y\in S_d$ such that the following conditions are satisfied:
\begin{enumerate}[label=(\roman*)]
\item For a point $y\in \bP^1_k$, if $n_1,\ldots,n_r$ denote the ramification degrees of the points in $X$ that map to $y$ under $f$ (so that $n_1+\cdots+n_r=d$), then $\sigma_y$ has cycle type $(n_1,\ldots,n_r)$;
\item $\prod_{y\in\bP^1_k}\sigma_y$ is the identity, where the product is taken with respect to the chosen order;
\item The Galois group of the Galois closure of the extension $k(X)\supseteq k(t)$ is generated by the permutations $\sigma_y$, for $y\in\bP^1_k$.
\end{enumerate}
Note that according to (i) the $\sigma_y$'s are trivial outside the branch locus of $f$, which is finite, and hence the product in (ii) makes sense.
\end{thm}
For our work we only need this result in the case $k=\overline{\bQ}$. Moreover, since $\bC$ is purely transcendental over $\bQ$, the relevant Galois groups do not change after base extension to $\bC$. So we can suppose for the purposes of our work that $k=\bC$ and work in the complex analytic category, regarding $X$ and $\bP^1_\bC$ as Riemann surfaces.
\begin{proof}[Proof in the case $k=\bC$]
This is essentially a particular case of the Proposition in section I of \cite{Harris}, which implies that the Galois group under consideration equals the monodromy group of $f$. Denote by $B=\{y_1,\ldots,y_n\}$ the branch locus of $f$. The fundamental group $\pi_1(\bP^1_\bC\setminus B)$ of the punctured sphere $\bP^1_\bC\setminus B$ is generated by the classes of $n$ loops around the points $y_1,\ldots,y_n$, which multiply to the identity. The monodromy group of $f$ is therefore generated by the images of those classes, which are permutations in $S_d$ which multiply the identity. We then just need to check that, if $n_1,\ldots,n_r$ denote the ramification degrees of the points in $X$ that map to $y_i$ under $f$, then the images of the loop around $y_i$ has cycle type $(n_1,\ldots,n_r)$. This is classical, and is established for example in \cite{Cavalieri}, Section 7.1.
\end{proof}

We take the opportunity to recall the well-known Riemann-Hurwitz formula.
\begin{thm}[Riemann-Hurwitz formula]
\label{rh}
Let $f:X\to Y$ be a morphism of projective algebraic curves over $k$, let $g_X$ be the genus of $X$, $g_Y$ the genus of $Y$, $d$ the degree of $f$ and $e_x$ the ramification degree of $f$ at the point $x\in X$. Then we have
$$2g_X-2=d(2g_Y-2)+\sum_{x\in X}(e_x-1)\text{.}$$
\end{thm}
\begin{proof}
See \cite{Cavalieri}, Theorem A.4.4.
\end{proof}

\section{The odd degree case}
\label{odd}
Throughout this section, let $n\geq 3$ be an odd positive integer. We borrow some ideas from \cite{Mes1990}, where the author uses them to solve the Inverse Galois Problem for $\tilde{A_n}$, the unique non-split central extension of $A_n$ by $\mathbb{Z}/2\mathbb{Z}$. The main claim is the following:

\begin{lemma}
\label{pq'-p'q}
Let $P\in\mathbb{Q}[X]$ be a separable monic polynomial of odd degree $n$. Then there exist nonzero polynomials $Q,R\in\mathbb{Q}[X]$, of degree at most $n-1$, such that
$$PQ'-P'Q=R^2\text{.}$$
\end{lemma}
The proof follows from the next two results. The first one essentially solves the problem of finding $Q$ and $R$ over the splitting field $L$ of $P$, in such a way that the obtained possible $R$'s form an $L$-vector space. Then the second result uses a standard Galois descent argument to obtain $Q$ and $R$ defined over $\mathbb{Q}$.
\begin{lemma}
\label{pq'-p'qw}
Let $P\in\mathbb{Q}[X]$ be a separable monic polynomial of odd degree $n$, and let $L$ be the splitting field of $P$ over $\mathbb{Q}$. Then there exists a nonzero vector subspace $W$ of $L[X]$, formed by polynomials of degree at most $n-1$, such that for every $R\in W$ there exists $Q\in L[X]$ with degree at most $n-1$ with
$$PQ'-P'Q=R^2\text{.}$$
Moreover, $W$ is invariant under the natural action of $\Gal(L/\mathbb{Q})$ on $\mathbb{Q}[X]$.
\end{lemma}
\begin{proof}
Write
$$P(X)=(X-\alpha_1)\ldots(X-\alpha_n)\text{,}$$
where $\alpha_1,\ldots,\alpha_n\in L$ are pairwise distinct. For each $k\in\{1,\ldots,n\}$, define
$$P_k(X)=\prod_{j\neq k}(X-\alpha_j)\text{.}$$
The polynomials $P_k(X)$ ($1\leq k\leq n$) are $L$-linearly independent. Indeed, if $\lambda_1,\ldots,\lambda_n\in L$ are such that
$$\lambda_1P_1+\cdots+\lambda_nP_n=0\text{,}$$
then evaluation at $\alpha_k$ yields $\lambda_k=0$. We now take our space $W$ to be
$$W=\left\{\sum_{j=1}^n\lambda_jP_j:\lambda_1,\ldots,\lambda_n\in L,\sum_{\substack{1\leq i\leq n\\i\neq j}}\frac{\lambda_i}{\alpha_i-\alpha_j}=0\text{ for }j=1,\ldots,n\right\}\text{.}$$
In order to show that $W$ is nonzero, it suffices, given the linear independence of the $P_k$'s, to show that the system of equations
$$\sum_{\substack{1\leq i\leq n\\i\neq j}}\frac{\lambda_i}{\alpha_i-\alpha_j}=0\text{ for }j=1,\ldots,n$$
has a nonzero solution. This is where the assumption that $n$ is odd proves crucial. In fact, this is a linear homogeneous equation with coefficient matrix
$$\begin{pmatrix}0&\frac{1}{\alpha_2-\alpha_1}&\cdots&\frac{1}{\alpha_n-\alpha_1}\\
\frac{1}{\alpha_1-\alpha_2}&0&\cdots&\frac{1}{\alpha_n-\alpha_2}\\
\vdots&\vdots&\ddots&\vdots\\
\frac{1}{\alpha_1-\alpha_n}&\frac{1}{\alpha_2-\alpha_n}&\cdots&0
\end{pmatrix}$$
and this matrix, being an anti-symmetric matrix of odd dimension, has determinant $0$, whence the above system has a non-trivial solution.

Let now $R=\sum\lambda_jP_j\in W$, where the $\lambda_j$ satisfy the above relation. We want to show that there exists $Q\in L[X]$, of degree at most $n-1$, such that $PQ'-P'Q=R^2$. We take
$$Q=-\sum_{j=1}^n\lambda_j^2P_j\text{.}$$
We now observe that we have the identity
$$P(P_i-P_j)=(\alpha_i-\alpha_j)P_iP_j\quad\text{when }i\neq j\text{.}$$
With this in mind, we compute
\begin{align}R^2&=\sum_{j=1}^n\lambda_j^2P_j^2+\sum_{\substack{1\leq i,j\leq n\\i\neq j}}\lambda_i\lambda_jP_iP_j\nonumber\\
&=\sum_{j=1}^n\lambda_j^2P_j^2+\sum_{\substack{1\leq i,j\leq n\\i\neq j}}\lambda_i\lambda_j\cdot\frac{P(P_i-P_j)}{\alpha_i-\alpha_j}\nonumber\\
&=\sum_{j=1}^n\lambda_j^2P_j^2-2P\sum_{j=1}^n\lambda_j\left(\sum_{\substack{1\leq i\leq n\\i\neq j}}\frac{\lambda_i}{\alpha_i-\alpha_j}\right)P_j\nonumber\\
&=\sum_{j=1}^n\lambda_j^2P_j^2\label{r2}\text{.}
\end{align}
On the other hand, we also have the identity
$$PP_j'=P_j\sum_{\substack{1\leq i\leq n\\i\neq j}}P_i$$
and hence
\begin{align}
PQ'&=\sum_{j=1}^n-\lambda_j^2PP_j\nonumber\\
&=\sum_{j=1}^n-\lambda_j^2P_j\sum_{\substack{1\leq i\leq n\\i\neq j}}P_i\nonumber\\
&=-\sum_{\substack{1\leq i,j\leq n\\i\neq j}}\lambda_j^2P_iP_j\label{pq'}
\end{align}
and moreover
\begin{align}
P'Q&=-\left(\sum_{i=1}^nP_i\right)\left(\sum_{j=1}^n-\lambda_j^2P_j\right)\nonumber\\
&=-\sum_{j=1}^n\lambda_j^2P_j^2-\sum_{\substack{1\leq i,j\leq n\\i\neq j}}\lambda_j^2P_iP_j\text{.}\label{p'q}
\end{align}
The computed expressions for $PQ'$ \eqref{pq'}, $P'Q$ \eqref{p'q} and $R^2$ \eqref{r2} imply directly that $PQ'-P'Q=R^2$, as desired.

Galois invariance of $W$ is straightforward.
\end{proof}

\begin{lemma}
\label{speiser}
Let $L/K$ be a finite Galois extension, and let $V\neq\{0\}$ be a vector space over $L$. Suppose that $G=\Gal(L/K)$ acts semi-linearly on $V$, in the sense that
$$\sigma(x+y)=\sigma(x)+\sigma(y)\quad\text{and}\quad\sigma(\lambda x)=\sigma(\lambda)\sigma(x)\quad\text{for all }x,y\in V\text{ and }\lambda\in L\text{.}$$
Then $V^G\neq\{0\}$. (Here $V^G$ denotes the set of elements of $V$ which are invariant under the action of $G$.)
\end{lemma}
\begin{proof}
Suppose $V^G=\{0\}$. Let $G=\{\sigma_1,\ldots,\sigma_n\}$. For every $v\in V$, $\sum_{j=1}^n\sigma_j(v)$ is invariant under the action of $G$, and hence
$$\sum_{j=1}^n\sigma_j(v)=0\text{.}$$
Let $(e_1,\ldots,e_n)$ be a basis of $L$ over $K$, and take any $v\neq0$ in $V$. What we have seen before implies that
\begin{equation}
\label{dep}
\sum_{j=1}^n\sigma_j(e_iv)=0\text{, i.e., }\sum_{j=1}^n\sigma_j(e_i)\sigma_j(v)=0
\end{equation}
for $i=1,\ldots,n$. But since, by a theorem of Dedekind, $\sigma_1,\ldots,\sigma_n$ are linearly independent over $K$, the vectors $(\sigma_j(e_1),\ldots,\sigma_j(e_n))$ for $j=1,\ldots,n$ are linearly independent over $K$, and the matrix $(\sigma_j(e_i))_{i=1,\ldots,n}^{j=1,\ldots,n}$ is invertible. Therefore \eqref{dep} implies that $\sigma_j(v)=0$ for $j=1,\ldots,n$. In particular, $v=0$, which is a contradiction.
\end{proof}
We are now in a position to prove Lemma \ref{pq'-p'q}.
\begin{proof}[Proof of Lemma \ref{pq'-p'q}]
Let $L$ and $W$ be as in Lemma \ref{pq'-p'qw}. The action of $\Gal(L/\mathbb{Q})$ on $W$ is clearly semi-linear, hence by Lemma \ref{speiser} it follows that $W^{\Gal(L/\mathbb{Q})}\neq\{0\}$. This means that there exists a nonzero $R\in W\cap\mathbb{Q}[X]$. We now want to find $Q$ with rational coefficients and degree at most $n-1$ such that $PQ'-P'Q=R^2$. This amounts to solving certain linear (non-homogeneous) equations, with rational coefficients, on the coefficients of $Q$ yet to determine. However we know that this system of equations has a solution over $L$, by definition of $W$. Since the equations are linear and have coefficients in $\mathbb{Q}$ this implies that the system also has a solution over $\mathbb{Q}$, as desired.
\end{proof}

The importance of Lemma \ref{pq'-p'q} for our work comes from the following result.

\begin{lemma}
\label{p-tq}
Let $P\in\mathbb{Q}[X]$ be monic irreducible of degree $n$, and suppose $Q,R\in\mathbb{Q}[X]$ have degree at most $n-1$ and $PQ'-P'Q=R^2$. Suppose furthermore that $\Delta(P)$ is a square. Then $\Delta(P-tQ)$ is a square for every $t\in\bQ$.
\end{lemma}
\begin{proof}
We will assume that $n$ is odd throughout the proof, since this is the main case of interest to us, but with slight sign variations the same proof works for even $n$. We use the description of discriminants in terms of resultants. Under the hypothesis of the lemma, we compute
\begin{align*}
&\Delta(P-tQ)\Res(P,Q)\\
&=(-1)^{\frac{n-1}{2}}\Res(P-tQ,P'-tQ')\Res(P,Q)&&\text{by Proposition \ref{resdisc}}\\
&=(-1)^{\frac{n-1}{2}}\Res(P-tQ,P'-tQ')\Res(P-tQ,Q)&&\text{by Proposition \ref{propres}(i),(ii)}\\
&=(-1)^{\frac{n-1}{2}}\Res(P-tQ,P'Q-tQ'Q)&&\text{by Proposition \ref{propres}(iii)}\\
&=(-1)^{\frac{n-1}{2}}\Res(P-tQ,PQ'-R^2-tQ'Q)\\
&=(-1)^{\frac{n-1}{2}}\Res(P-tQ,Q'(P-tQ)-R^2)\\
&=(-1)^{\frac{n-1}{2}}\Res(P-tQ,-R^2)&&\text{by Proposition \ref{propres}(ii)}\\
&=(-1)^{\frac{n+1}{2}}\Res(P-tQ,R^2)\\
&=(-1)^{\frac{n+1}{2}}\Res(P-tQ,R)^2&&\text{by Proposition \ref{propres}(iii)}\text{.}\\
\end{align*}

We conclude that $(-1)^{\frac{n+1}{2}}\Res(P,Q)\Delta(P-tQ)$ is a square for all $t$. In particular, by setting $t=0$, this implies that $(-1)^{\frac{n+1}{2}}\Res(P,Q)\Delta(P)$ is a square. Since $\Delta(P)$ is a (nonzero) square by assumption it follows that $(-1)^{\frac{n+1}{2}}\Res(P,Q)$ is a square. But then the fact that $(-1)^{\frac{n+1}{2}}\Res(P,Q)\Delta(P-tQ)$ is a square implies that $\Delta(P-tQ)$ is a square (note that $\Res(P,Q)\neq0$: since $P$ is irreducible of degree $n$ and $Q$ has degree at most $n-1$, $P$ and $Q$ cannot have a common root).
\end{proof}

We are finally ready to prove Theorem \ref{oddcase}.

\begin{proof}[Proof of Theorem \ref{oddcase}]
Take \emph{any} monic polynomial $P$, of degree $n$, with integer coefficients and Galois group $A_n$. That such a polynomial exists is standard; see for example \cite{NV1983}, Corollary 2.4. By Lemma \ref{pq'-p'q} we can find $Q,R\in\mathbb{Q}[X]$ with degrees at most $n-1$ such that $PQ'-P'Q=R^2$. By appropriately rescaling $Q$ and $R$ we can actually assume that $Q,R\in\mathbb{Z}[X]$.

We now recall a famous theorem of Dedekind (\cite{Cox2012}, Theorem 13.4.5) which states that, if $f\in\mathbb{Z}[X]$ is a monic irreducible polynomial and $p$ is a prime such that the reduction of $f$ modulo $p$ is separable (in other words, such that $p$ does not divide $\Delta(f)$) and $f\, (\mathrm{mod}\text{ } p)$ factors in $\mathbb{F}_p[X]$ into irreducibles of degrees $d_1,\ldots,d_k$, then the Galois group of $f$ over $\mathbb{Q}$ contains an element which, viewed as a permutation of the roots of $f$, has cycle type $(d_1,\ldots,d_k)$. Such an element is obtained by considering the Frobenius element corresponding to some prime ideal lying over $p$ in the splitting field $L$ of $f$. Therefore, by Chebotarev's density theorem, given $d_1,\ldots,d_k$ such that $\Gal(L/\mathbb{Q})$ contains a permutation of cycle type $(d_1,\ldots,d_k)$, there exist infinitely many primes $p$ such that $f\, (\mathrm{mod}\text{ } p)$ factors in $\mathbb{F}_p[X]$ into irreducibles of degrees $d_1,\ldots,d_k$.

Let now $\ell$ be a prime with $\frac{n}{2}<\ell\leq n$, which exists by Bertrand's Postulate. Since the Galois group of $P$ is $A_n$ and $\ell$ is odd, the Galois group of $P$ contains an $\ell$-cycle. By our remarks above there exists one prime $p$ (in fact, infinitely many) such that the reduction $P\, (\mathrm{mod}\text{ }p)$ factors in $\mathbb{F}_p[X]$ as a product of one irreducible polynomial of degree $\ell$ and $n-\ell$ polynomials of degree $1$. Similarly, there exists a prime $q$ such that $P\, (\mathrm{mod}\text{ }q)$ factors in $\mathbb{F}_q[X]$ as a product of one irreducible polynomial of degree $3$ and $n-3$ polynomials of degree $1$, and there exists a prime $r$ such that $P\, (\mathrm{mod}\text{ }r)$ is itself irreducible in $\mathbb{F}_r[X]$. We now consider polynomials of the form
$$f=P-pqrNQ,\quad\text{with $N$ integer.}$$
We claim that all these have Galois group $A_n$. By reducing modulo $r$ we see that $f$ is irreducible. Modulo $p$, the polynomial $f$ factors as a product of an irreducible polynomial of degree $\ell$ and $n-\ell$ irreducible polynomials of degree $1$. Therefore, by Dedekind's Theorem, the Galois group of $f$ over $\mathbb{Q}$ contains an $\ell$-cycle. By the same reasoning carried out with the prime $q$ we see that the Galois group of $f$ over $\mathbb{Q}$ contains a $3$-cycle. Now a transitive subgroup of $S_n$ containing an $\ell$-cycle and a $3$-cycle in $S_n$ must contain the alternating group $A_n$; in fact, the containment of an $\ell$-cycle implies that this subgroup must primitive, and then a famous result of Jordan asserts that a primitive subgroup of $S_n$ containing a $3$-cycle must contain $A_n$ (see for example ``Theorem (Jordan)'' in \cite{Isaacs}). So the Galois group of $f$ contains $A_n$. Since $\Delta(f)$ is a square by Lemma \ref{p-tq}, it follows from Theorem \ref{sqdisc} that the Galois group of $f$ equals $A_n$. Hence we obtain the desired line of polynomials.
\end{proof}

\begin{example}
We use the strategy of the proof above to construct a family of monic polynomials of degree $5$ in $\mathbb{Z}[X]$ with $\gg H$ polynomials having coefficients in $[-H,H]$. The polynomial
$$P(X)=X^5+4X^4-5X^3-28X^2-18X-2$$
can be shown to have Galois group $A_5$. The polynomials
$$Q(X)=-49X^4-287X^3-474X^2-191X-7$$
and
$$R(X)=7X^4+41X^3+81X^2+47X+16$$
are such that $PQ'-P'Q=R^2$. Therefore, by Lemma \ref{p-tq}, the discriminant of $P-tQ$ is a square for all $t$. In fact, one can check that
$$\Delta(P-tQ)=(22068963t^4-6897879t^3+946647t^2-62469t+1762)^2\text{.}$$
In the proof above we can now take, say, $\ell=5$. We find that $P$ is irreducible in $\mathbb{F}_7[X]$, so we can take $p=7$, and $r=7$ as well. On the other hand we see that $P$ factors in $\mathbb{F}_{17}[X]$ as
$$(X^3+10X^2+8X+9)(X+6)(X+5)$$
and $X^3+10X^2+8X+9$ is irreducible in $\mathbb{F}_{17}[X]$. Therefore all polynomials of the form
\begin{align*}
f(X)&=P(X)-7\times 7\times 17NQ(X)\\
&=X^5+(40817N+4)X^4+(239071N-5)X^3\\
&+(394842N-28)X^2+(159103N-18)X+(5831N-2)
\end{align*}
have Galois group $A_5$.
\end{example}

\begin{obs}
\label{obs}
We remark here that the strategy of the final part of the proof of Theorem \ref{oddcase} actually yields the following. Suppose we have a line of polynomials of the form \eqref{lines}, such that the discriminant of any polynomial in the line is a square and \emph{at least one} polynomial in the line has Galois group $A_n$. Then we may find a ``sub-line'' whose polynomials all have Galois group $A_n$.
\end{obs}

\section{The even degree case}
\label{even}
In this section we prove Theorems \ref{n-1square} and \ref{evencase}. The proof hinges on the following result.

\begin{lemma}
\label{pqr}
Let $n\geq3$ be a positive integer, let $k$ be a field of characteristic $0$ and let $p,q,r,p',q',r'$ be $k$-rational points in $\bP^1_k$, such that $p,q,r$ are pairwise distinct and $p',q',r'$ are also pairwise distinct. Then there exists a unique morphism $f:\bP^1_k\to\bP^1_k$ of degree $n$ such that the fiber above $p'$ contains the point $p$ with multiplicity $n-1$, the fiber above $q'$ contains the point $q$ with multiplicity $n-1$ and the fiber above $r'$ contains the point $r$ with multiplicity $3$.
\end{lemma}
\begin{proof}
Since the automorphism group $\Aut_k(\bP^1_k)$ acts $3$-transitively on $\bP^1_k$, we may assume without loss of generality that $(p,q,r)=(p',q',r')=(0,\infty,1)$. Any morphism $f:\bP^1_k\to\bP^1_k$ of degree $n$ can be given explicitly as
$$f([x:y])=[P(x,y):Q(x,y)]$$
where $P,Q$ are relatively prime homogeneous polynomials of degree $n$ in $k[x,y]$. The condition that the fiber above $0$ contains $0$ with multiplicity $n-1$ prescribes that $P(x,y)=x^{n-1}(ax+by)$ for some $a,b\in k$ with $b\neq0$. Similarly the condition that the fiber above $\infty$ contains $\infty$ with multiplicity $n-1$ prescribes that $Q(x,y)=y^{n-1}(cx+dy)$ for some $c,d\in k$ with $c\neq0$. It is easy to see that the coprimality condition is now equivalent to $a$, $d$ and $ad-bc$ all being nonzero.

We must now show that there exists exactly one valid choice of $a,b,c,d\in k$, up to scaling, such that the fiber of the resulting map $f$ above $1$ contains the point $1$ with multiplicity $3$. The condition that the fiber above $1$ contains the point $1$ with multiplicity \emph{at least} $3$ is equivalent to requiring that
$$(x-y)^3\quad\text{divides}\quad P(x,y)-Q(x,y)\text{.}$$
Firstly, the condition that $f$ maps $1$ to $1$ is equivalent to having the polynomial identity $P(x,x)=Q(x,x)$, which translates into
\begin{equation}
\label{lineq1}
a+b=c+d\text{.}
\end{equation}
If the above holds, then setting $P'(x,y)=\frac{\partial}{\partial x}P(x,y)$ and $Q'(x,y)=\frac{\partial}{\partial x}Q(x,y)$ the condition that $(x-y)^2$ divides $P(x,y)-Q(x,y)$ is equivalent to the polynomial identity $P'(x,x)=Q'(x,x)$. One computes that $P'(x,y)=nax^{n-1}+(n-1)bx^{n-2}y$ and $Q'(x,y)=cy^{n-1}$, so that the identity $P'(x,x)=Q'(x,x)$ becomes equivalent to
\begin{equation}
\label{lineq2}
na+(n-1)b=c\text{.}
\end{equation}
Finally, if the above two equations hold, setting $P''(x,y)=\frac{\partial^2}{\partial x^2}P(x,y)$ and $Q''(x,y)=\frac{\partial^2}{\partial x^2}Q(x,y)$, the condition that $(x-y)^3$ divides $P(x,y)-Q(x,y)$ is equivalent to the polynomial identity $P''(x,x)=Q''(x,x)$. One computes that  $P''(x,y)=n(n-1)ax^{n-2}+(n-1)(n-2)bx^{n-3}y$ and $Q''(x,y)=0$, so that the identity $P''(x,x)=Q''(x,x)$ translates into
\begin{equation}
\label{lineq3}
na+(n-2)b=0\text{.}
\end{equation}
Now the solution set of the system formed by the equations \eqref{lineq1},\eqref{lineq2} and \eqref{lineq3} is one-dimensional, spanned by $(a,b,c,d)=(n-2,-n,-n,n-2)$, establishing the existence and uniqueness of the desired morphism $f$. That the multiplicity of $1$ on the fiber above $1$ of the obtained morphism is \emph{exactly} $3$ can be seen by differentiating once more, but we note that it follows from the Riemann-Hurwitz formula (Theorem \ref{rh}). Indeed, keeping the notation from Theorem \ref{rh} and keeping in mind that $\bP^1$ has genus $0$, the points $0$ and $\infty$ contribute $n-2$ each to the sum $\sum_{x\in\bP^1}(e_x-1)$, and the point $1$ contributes at least $2$. But the Riemann-Hurwitz formula tells us that this sum is precisely $2n-2=(n-2)+(n-2)+2$, so that there can be no further contributions and the multiplicity of $1$ on the fiber above $1$ is exactly $3$.
\end{proof}

We can now prove Theorem \ref{evencase}.
\begin{proof}[Proof of Theorem \ref{evencase}]
Motivated by Lemma \ref{pqr}, for an even integer $n>3$ we consider the unique morphism $f:\bP^1_\bQ\to\bP^1_\bQ$ such that the fibers above $0$, $\infty$ and $(n-2)^{n-2}$ contain the points $0$, $\infty$ and $n-2$ with multiplicities $n-1$, $n-1$ and $3$, respectively. We observe that, as explained at the end of the proof of Lemma \ref{pqr}, the Riemann-Hurwitz formula implies that this morphism is actually unramified at any point other than $0$, $\infty$ and $n-2$. We will regard it as a morphism from $\bP^1_{\overline{\bQ}}$ to $\bP^1_{\overline{\bQ}}$ that happens to be defined over $\bQ$. It is straightforward to compute this morphism using the work we have done in the proof of Lemma \ref{pqr}. We first post-compose the morphism that came out of that proof with the automorphsim $[x:y]\to\left[\frac{x}{n-2}:y\right]$ of $\bP^1_\bQ$. Note that this automorphism maps the triple $(0,\infty,n-2)$ to $(0,\infty,1)$. We then pre-compose it with the automotphism $[x:y]\to[(n-2)^{n-2}x:y]$, and we obtain the following morphism (given as a rational function on the affine chart corresponding to the standard inclusion $\bbA^1\subseteq\bP^1$):
$$f(x)=\frac{x^n-nx^{n-1}}{-nx+(n-2)^2}\text{.}$$
Under the curves-fields correspondence, this morphism corresponds to the extension of $\overline{\bQ}(T)$ obtained by adjoining a root of
$$F(X)=X^n-nX^{n-1}-T(-nX+(n-2)^2)\text{.}$$ We obtain from Theorem \ref{cover} that the Galois group of this polynomial over $\overline{\bQ}(T)$ is a subgroup of $S_n$ generated by two $(n-1)$-cycles and a $3$-cycle. Such a subgroup is necessarily contained in $A_n$ (here we used that $n$ is even). On the other hand, being the Galois group of an irreducible polynomial, it must be a transitive subgroup of $S_n$, and a transitive subgroup of $S_n$ containing an $(n-1)$-cycle must be primitive. Since it contains a $3$-cycle, a famous result of Jordan (see for example ``Theorem (Jordan)'' in \cite{Isaacs}) now implies that $P(X)$ has Galois group $A_n$ over $\overline{\bQ}(T)$.

We now prove that it actually has Galois group $A_n$ over $K(T)$ for some at most extension $K$ of $\bQ$ that has degree at most $2$. We observe here that the Galois group of $F(X)$ over $K(T)$ can at most \emph{decrease} when we extend the field $K$. Indeed, if $K\subseteq L$ are algebraic extensions of $\bQ$ and $S$ is the splitting field of $F(X)$ over $K(T)$, then the splitting field of $F(X)$ over $L(T)$ is the compositum of $S$ and $L(T)$, and it is a general fact that if $E\subseteq F$ is a Galois extension of fields and $G$ is an extension of $E$ then there is an injection
\begin{center}
\begin{tikzcd}
\mathrm{Gal}(FG/G) \arrow[r, hook] & \mathrm{Gal}(F/E)\text{.}
\end{tikzcd}
\end{center}
So for any number field $K$ the Galois group of $F(X)$ over $K(T)$ is either $S_n$ or $A_n$, and we can distinguish between the two using Theorem \ref{sqdisc}. Since the Galois group of $F(X)$ over $\overline{\bQ}(T)$ is $A_n$, the discriminant of $F(X)$ is a square in $\overline{\bQ}(T)$. However this discriminant is also a polynomial in $T$ with rational coefficients. This implies that this discriminant is of the form $ad(T)^2$ for some polynomial in $d(T)\in\bQ(T)$. But this is clearly a square in $K(T)$ for $K=\bQ(\sqrt{a})$, so for this choice of $K$ the Galois group of $F(X)$ over $K(T)$ is $A_n$.

By Hilbert's irreducibility Theorem, $F(X)$ yields a line containing at least one polynomial with Galois group $A_n$, and Observation \ref{obs} completes the proof.
\end{proof}

In order to prove Theorem \ref{n-1square}, we will compute explicitly the discriminant of $F(X)$ over $\bQ(T)$ and show that, if $n-1$ is a square, we can take $a=1$ in the above proof (i.e. the discriminant of $F(X)$ is actually a square in $\bQ(T)$), and hence we can take $K=\bQ$ in the proof of Theorem \ref{evencase}.

\begin{proof}[Proof of Theorem \ref{n-1square}]
Let
$$P(X)=X^n-nX^{n-1}\quad\text{and}\quad Q(X)=-nX+(n-2)^2\text{.}$$
We will prove that, if $n-1$ is an odd square, then $\Delta(P-TQ)$ is a square in $\bQ(T)$, which, as explained before, implies that we can take $a=1$ (in other words, $K=\bQ$) in the proof of Theorem \ref{evencase}. For that we compute, similarly to what we did in the proof of Lemma \ref{p-tq}, that
\begin{align*}
&\Delta(P-TQ)\Res(P,Q)\\
&=(-1)^{\frac{n}{2}}\Res(P-TQ,P'-TQ')\Res(P,Q)&&\text{by Proposition \ref{resdisc}}\\
&=(-1)^{\frac{n}{2}}\Res(P-TQ,P'-TQ')\Res(P-TQ,Q)&&\text{by Proposition \ref{propres}(i),(ii)}\\
&=(-1)^{\frac{n}{2}}\Res(P-TQ,P'Q-TQ'Q)&&\text{by Proposition \ref{propres}(iii)}\\
&=(-1)^{\frac{n}{2}}\Res(P-TQ,Q'(P-TQ)+P'Q-PQ')\\
&=(-1)^{\frac{n}{2}}\Res(P-tQ,P'Q-PQ')\text{.}&&\text{by Proposition \ref{propres}(ii)}\\
\end{align*}
We compute directly that
$$P'Q-PQ'=-n(n-1)X^{n-2}(X-(n-2))^2$$
and therefore
\begin{align*}
&\Delta(P-TQ)\Res(P,Q)\\
&=(-1)^{\frac{n}{2}}\Res(P-TQ,-n(n-1)X^{n-2}(X-(n-2))^2)\\
&=(-1)^{\frac{n}{2}}n^n(n-1)^n\Res(P-TQ,X)^{n-2}\Res(P-TQ,X-(n-2))^2\\
&=(-1)^{\frac{n}{2}}n^n(n-1)^n(P(0)-TQ(0))^{n-2}(P(n-2)-TQ(n-2))^2\\
&=4(-1)^{\frac{n}{2}}n^n(n-1)^n(n-2)^{2n-2}T^{n-2}(T-(n-2)^{n-2})^2\text{.}\\
\end{align*}
On the other hand, we also compute that
\begin{align*}
\Res(P,Q)&=\Res(X^{n-1}(X-n),-nX+(n-2)^2)\\
&=\Res(X,-nX+(n-2)^2)^{n-1}\Res(X-n,-nX+(n-2)^2)\\
&=-4(n-1)(n-2)^{2n-2}\text{.}
\end{align*}
Combining the last two equalities it follows that
\begin{equation}
\label{discp-tq}
\Delta(P-TQ)=(-1)^{\frac{n}{2}+1}(n-1)^{n-1}n^nT^{n-2}(T-(n-2)^{n-2})^2\text{.}
\end{equation}
Now the above is clearly a square in $\bQ(T)$ if $n-1$ is an odd square, taking into account that if $n-1$ is an odd square then $n\equiv2\pmod{4}$ and therefore $(-1)^{\frac{n}{2}+1}=1$.
\end{proof}

\section{The quartic case}
\label{quarticsec}
In this section we present our work on the case of polynomials of degree $4$, with the aim of establishing Theorems \ref{class4}, \ref{nostrong} and \ref{hm}. We begin with some general considerations. We have pointed out before, in the proof of Theorem \ref{evencase}, that the Galois group, over $K(T)$, of a polynomial with coefficients in $\bQ(T)[X]$ either remains the same or decreases when we increase the algebraic extension $K$ of $\bQ$. So if $P$ and $Q$ are polynomials in $\bQ[X]$ such that $P-TQ$ defines an alternating line of quartic polynomials, the Galois group of $P-TQ$ over $\overline{\bQ}(T)$ is a subgroup of $A_4$. Of course, by definition, it actually equals $A_4$ if the corresponding line is strong; see the Introduction for the meaning of the dichotomy strong/weak lines.

Given such a line, we may associate to it the rational function $\frac{P}{Q}$, which we may then regard as a morphism $f$ (of degree $4$) of the projective line $\bP^1_\bQ$, or as a morphism of $\bP^1_{\overline{\bQ}}$ which happens to be defined over $\bQ$. The Galois group of $P-TQ$ over $\overline{\bQ}(T)$ is then the Galois group of the Galois closure of the field extension corresponding to the morphism of curves $f$. We may then use Theorem \ref{cover} to relate this Galois group to the geometry of $f$.

Since $f$ has degree $4$, the Riemann-Hurwitz formula (Theorem \ref{rh}) gives us
\begin{equation}
\label{rh4}
\sum_{x\in\bP^1_{\overline{\bQ}}}(e_x-1)=6\text{.}
\end{equation}
We rewrite the above sum as
\begin{equation}
\label{rh4v}
\sum_{y\in\bP^1_{\overline{\bQ}}}\sum_{f(x)=y}(e_x-1)=\sum_{y\in\bP^1_{\overline{\bQ}}}(4-\#f^{-1}(y))
\end{equation}
and observe that, in the notation of Theorem \ref{cover}, the tuples $(d_1,\ldots,d_k)$ corresponding to points $y\in\bP^1_{\overline{\bQ}}$ must correspond to cycle decompositions of \emph{even} permutations. The admissible tuples are therefore, apart from the empty tuple, $(3,1)$ and $(2,2)$, where for simplicity we ignore the ordering within each tuple. This shows in particular that all (nonzero) summands in \eqref{rh4v} equal $2$ and therefore, by \eqref{rh4}, there are exactly three such nonzero summands; in other words, $f$ admits exactly $3$ non-reduced fibers.

We can readily exclude the case in which one of these non-reduced fibers has ramification pattern $(3,1)$ and the other two have ramification pattern $(2,2)$. This is simply because there are no permutations $\sigma_1,\sigma_2,\sigma_3\in S_4$, with $\sigma_1$ having cycle type $(3,1)$ and $\sigma_2,\sigma_3$ having cycle type $(2,2)$, whose product is the identity, in contradiction with Theorem \ref{cover}. Moreover, we observe that if all three non-reduced fibers have ramification pattern $(2,2)$, then the Galois group of $P-TQ$ over $\overline{\bQ}$ is strictly smaller than $A_4$; indeed all permutations of cycle type $(2,2)$ in $S_4$ lie in an order $4$ subgroup isomorphic to the Klein four-group.

\bigskip

To sum up, given an alternating line of polynomials of degree $4$, we can associate to it a morphism $f:\bP^1_{\overline{\bQ}}\to\bP^1_{\overline{\bQ}}$, defined over $\bQ$, that ramifies above exactly $3$ points in the target, and such one of the following holds:
\begin{itemize}
\item\textbf{Case 1:} All three non-reduced fibers have ramification pattern $(3,1)$;
\item\textbf{Case 2:} Two non-reduced fibers have ramification pattern $(3,1)$, while the remaining one has ramification pattern $(2,2)$;
\item\textbf{Case 3:} All three non-reduced fibers have ramification pattern $(2,2)$.
\end{itemize}
Moreover, Case 3 cannot occur if the line giving rise to $f$ is a strong alternating line.

\bigskip

We now prove a Lemma that lays the groundwork for the proofs of \ref{nostrong} and Theorem \ref{hm}.

\begin{lemma}
\label{prenostrong}
There exists no morphism $f:\bP^1_{\overline{\bQ}}\to\bP^1_{\overline{\bQ}}$, defined over $\bQ$, that ramifies above exactly $3$ points in the target, fitting in Case 1 or Case 2 as described above.
\end{lemma}

\begin{proof}
Assume for the sake of contradiction that there exists a morphism with the aforementioned properties. We start with Case 1. Then there exist points $p_i,p_i'\in\bP^1_{\overline{\bQ}}$ such that the fiber above $p_i'$ contains the point $p_i$ with multiplicity $3$ for $i=1,2,3$. Denote by $L$ the minimal field of definition of the six points $p_i,p_i'$. By Lemma \ref{pqr} the morphism $f$ is uniquely determined by the points $p_i,p_i'$, and since all choices of $p_i,p_i'$ are equivalent up to automorphisms of $\bP^1_K$ the field extension of $K(T)$ corresponding to $f$ does not depend on the choice of $p_i,p_i'$. The computation in the proof of Theorem \ref{n-1square} (namely, equality \eqref{discp-tq}) shows that the discriminant of this unique field extension is $-3$, regarded as an element of $L^\times/L^{\times2}$. Since, on the other hand, this discriminant is a square, it follows that $-3$ is a square in $L$.

We now observe that the absolute Galois group $G_\bQ=\Gal(\overline{\bQ}/\bQ)$ acts on the set $\{\{p_1,p_1'\},\{p_2,p_2'\},\{p_3,p_3'\}\}$ by permutations. In other words, if $\sigma\in G_\bQ$ is such that $\sigma(p_i')=p_j'$, then $\sigma(p_i)=p_j$. This follows from the fact that $\sigma(p_i)$ must be a point that is contained with multiplicity $3$ on the fiber above $p_j'$, and the only such point is $p_j$. If this action fixes one point, then that point is defined over $\bQ$. Otherwise, each of the three points $p_1',p_2',p_3'$ has exactly three conjugates in $\bP^1_{\overline{\bQ}}$, and therefore each of them is defined over a cubic extension of $\bQ$ (which may not be the same for all three points).

We can therefore assume without loss of generality that $p_1'$ is defined over some field $K$ with $[K:\bQ]=3$ (we are \emph{not} stating that $K$ is the minimal field of definition of $p_1'$; it could also be $\bQ$ itself). Now every element of $\Gal(\overline{\bQ}/K)$ permutes $\{p_2,p_2'\}$ and $\{p_3,p_3'\}$, so that these two form a complete set of conjugates over $K$ and the minimal field of definition of the six points $p_i,p_i'$ (which we called $L$ above) is an extension of $K$ of degree at most $2$. On the other hand we showed that $L$ contains $\sqrt{-3}$. Since $\sqrt{-3}\notin K$, as $2$ does not divide $[K:\bQ]$, it follows that $L=K(\sqrt{-3})$.

To sum up, we showed that $p_1,p_1'$ are defined over $K$ and $p_2,p_2',p_3,p_3'$ are defined over $K(\sqrt{-3})$, with $\{p_2,p_2'\}$ and $\{p_3,p_3'\}$ being conjugates over $K$. Regarding $f$ now as a morphism of $\bP^1_K$, then $f$ corresponds to a field extension of $K(T)$, the discriminant of which is a square in $K(T)$. We will show that the properties stated in this paragraph are incompatible, thereby arriving at the desired contradiction.

For that we observe that any point of $\bP^1$ defined over $K(\sqrt{-3})$ but not over $K$ is equivalent to $\sqrt{-3}$ itself under an automorphism of $\bP^1$ defined over $K$ (here we are representing points of $\bP^1$ by their coordinate with respect to the standard inclusion of $\bbA^1$; as usual we denote by $\infty$ the point which is not covered by this inclusion). In fact any such point can be written as $s+u\sqrt{-3}$, for some $s,u\in K$ with $u\neq0$, whence it can be mapped to $\sqrt{-3}$ by the rational function $x\mapsto\frac{x-s}{u}$, which defines an automorphism of $\bP^1$ over $K$. Therefore we can assume without loss of generality that $p_2=p_2'=\sqrt{-3}$; changing $f$ by automorphisms of $\bP^1$ defined over $K$ does not change the corresponding field extension of $K(T)$. Note that this will force that $p_3=p_3'=-\sqrt{-3}$.

Now let $f$ be defined in homogeneous coordinates by
$$f([x:y])=[P(x,y):Q(x,y)]$$
where $P,Q\in K[x,y]$ are homogeneous, relatively prime quartic polynomials. The condition that $\sqrt{-3}$ lies with multiplicity $3$ on the fiber of $f$ above $\sqrt{-3}$ means that
$$(x-\sqrt{-3}y)^3\text{ divides } P(x,y)-\sqrt{-3} Q(x,y)\text{.}$$
The quotient between the right hand side and the left hand side is a linear form in $K[x,y]$, which can therefore be written as $L_1(x,y)+\sqrt{-3}L_2(x,y)$ where $L_1,L_2\in K[x,y]$ are linear forms. We therefore obtain
\begin{align*}
P(x,y)-\sqrt{-3}Q(x,y)&=(x-\sqrt{-3}y)^3(L_1(x,y)+\sqrt{-3}L_2(x,y))\\
&=(x^3-9xy^2)L_1(x,y)-9(y^3-x^2y)L_2(x,y)\\&+((x^3-9xy^2)L_2(x,y)+3(y^3-x^2y)L_1(x,y))\sqrt{-3}\text{.}
\end{align*}
It follows that
$$P(x,y)=(x^3-9xy^2)L_1(x,y)-9(y^3-x^2y)L_2(x,y)$$
and
$$Q(x,y)=-(x^3-9xy^2)L_2(x,y)-3(y^3-x^2y)L_1(x,y)\text{.}$$
Letting $L_1(x,y)=ax+by$ and $L_2(x,y)=cx+dy$, we remark that $ad-bc\neq0$, otherwise the forms $L_1$ and $L_2$ are not linearly independent and by the equalities above $P$ and $Q$ share a common factor. From the equalities above, $f$ corresponds to the rational function given in affine coordinates by
\begin{equation}
\label{morph}
f(x)=\frac{(x^3-9x)(ax+b)-9(1-x^2)(cx+d)}{-(x^3-9x)(cx+d)-3(1-x^2)(ax+b)}\text{.}
\end{equation}
One computes that
$$f'(x)=\frac{(x^2+3)^2}{(-(x^3-9x)(cx+d)-3(1-x^2)(ax+b))^2}\cdot(Ax^2+Bx+C)$$
where
\begin{align*}
&A=3a^2-ad+bc+9c^2\\
&B=6ab+18cd\\
&C=-3ad+3b^2+3bc+9d^2\text{.}
\end{align*}
The fact that $(x^2+3)^2$ is a factor of $f'$ indicates precisely that $f$ ramifies at $\pm\sqrt{-3}$ with multiplicity $3$. This already excludes the possibility that $\pm\sqrt{-3}$ are roots of $Ax^2+Bx+C$ (otherwise they would have higher ramification index) and, if this quadratic polynomial has two distinct roots, these correspond to two further ramification points of $f$, which we know do not exist. Therefore $Ax^2+Bx+C$ has a double root, meaning that $B^2-4AC=0$, and it turns out that
$$B^2-4AC=12(ad-bc)(3a^2-10ad+3d^2+b^2+10bc+9c^2)\text{.}$$
We have remarked before that $ad-bc\neq0$, and therefore
\begin{equation}
\label{quadric}
3a^2-10ad+3d^2+b^2+10bc+9c^2=0\text{.}
\end{equation}
One must therefore understand rational points on the surface in $\bP^3_K$ defined by \eqref{quadric}. This turns out to be particularly simple, as this quadric surface turns out to be isomorphic (over $K$) to $\bP^1_K\times\bP^1_K$. In fact, one may rewrite \eqref{quadric} as
$$(b+9c)(b+c)-(3d-a)(3a-d)=0$$
and we realize that under a linear invertible change of coordinates in $\bP^3$ our surface becomes the familiar Segre embedding of $\bP^1\times\bP^1$ into $\bP^3$. Solutions to the above are given by setting
$$(b+9c,b+c,3d-a,3a-d)=(pq,rs,pr,qs)$$
for some $p,q,r,s\in K$, which in turn gives
\begin{equation}
\label{abcd}
(a,b,c,d)=\left(\frac{pr+3qs}{8},\frac{9rs-pq}{8},\frac{pq-rs}{8},\frac{3pr+qs}{8}\right)\text{.}
\end{equation}
We now recall that, in light of the expression \eqref{morph} for $f$, the discriminant of the extension of $K(T)$ corresponding to $f$ is the discriminant of the polynomial $P^{\mathrm{ih}}-TQ^{\mathrm{ih}}$, where $P^\mathrm{ih}$ and $Q^\mathrm{ih}$ are the de-homogenizations of $P$ and $Q$ (the numerator and denominator of \eqref{morph}). Computing this discriminant and replacing $a,b,c,d$ according to \eqref{abcd} one arrives at
$$\Delta(P^\mathrm{ih}-TQ^\mathrm{ih})=-\frac{27}{64}(p^2+3s^2)^2(T^2+3)^2R(T)^2$$
with
\begin{align*}
R(T)&=9pr^3-9pq^2r+27qr^2s-3q^3s+(9pqr^2-pq^3-9r^3s+9q^2rs)T\text{.}
\end{align*}
Since $-3$ is not a square in $K$, this is not a square in $K(T)$, which is the desired contradiction. We observe that the argument we used shows in fact that strong alternating lines giving rise to Case 1 do not exist over any number field in which $-3$ is not a square. However, the method used in the proof allows us to construct many examples of strong alternating lines over $\bQ(\sqrt{-3})$.

We turn now to Case 2. Again we look at the action of the absolute Galois group $G_\bQ$ on the fibers of our morphism $f$. Since this action preserves the ramification pattern of the fibers, the (unique) fiber with ramification pattern $(2,2)$ is fixed and is therefore defined over $\bQ$, i.e. it is the fiber above some point in $\bP^1_{\overline{\bQ}}$ defined over $\bQ$. This point is therefore mapped to $\infty$ under an automorphism of $\bP^1_{\overline{\bQ}}$ defined over $\bQ$, and hence we may suppose without loss of generality that it \emph{is} $\infty$.

The points with multiplicity $2$ on the fiber of $f$ above $\infty$ are then also permuted by the action of $G_\bQ$, and are therefore either defined over $\bQ$ or conjugates of one another defined over some quadratic extension of $\bQ$, say $\bQ(\sqrt{d})$; this also includes the case when they are defined over $\bQ$, since we may then take $d=1$. Without loss of generality we may suppose that they are actually $\sqrt{d}$ and $-\sqrt{d}$, since any point in $\bP^1_{\overline{\bQ}}$ defined over $\bQ(\sqrt{d})$ is equivalent to $\sqrt{d}$ up to an automorphism defined over $\bQ$. In terms of the representation of $f$ as a rational function on $\bbA^1$ this means that the denominator of $f$ is $(x^2-d)^2$.

Once we differentiate $f$, we obtain therefore a rational function whose denominator is $(x^2-d)^3$, and moreover its zeros must correspond to the ramification points of $f$ that do not lie over $\infty$. Since each of these has multiplicity $3$, the corresponding factors of $f'$ have multiplicity $2$, and therefore the numerator of $f'$ is a square. We may therefore suppose that
$$f'(x)=\frac{(ax^2+bx+c)^2}{(x^2-d)^3}$$
for some $a,b,c\in\bQ$.

Regarding $f$ as a meromorphic function on $\bC$, we compute that
$$\mathrm{Res}_{\sqrt{d}}f'=\frac{3d^2a^2-2dac-db^2+3c^2}{16d^2\sqrt{d}}\text{,}$$
but on the other hand the above, being a residue of the derivative of a meromorphic function, must equal $0$. We therefore conclude that
$$3d^2a^2-2dac-db^2+3c^2=0\text{.}$$
From here we see first that $b\neq0$. In fact, if $b=0$, then the above becomes equivalent to
$$(3c-da)^2+8d^2a^2=0$$
which is plainly impossible. We may therefore assume without loss of generality that $b=1$, since this amounts to dividing $f$ by $b^2$. We introduce the quadratic form
$$q(x,y)=3x^2-2xy+3y^2$$
and observe that, with $u=ad$, the previous equality becomes equivalent to
\begin{equation}
\label{quc}
d=q(u,c)\text{.}
\end{equation}
We see therefore that
$$f'(x)=\frac{\left(\frac{u}{d}x^2+x+c\right)^2}{(x^2-d)^3}=\frac{\left(\frac{u}{q(u,c)}x^2+x+c\right)^2}{(x^2-q(u,c))^3}$$
and one checks that the above is the derivative of the rational function
$$\frac{-2u^2x^3-2uq(u,c)x^2-2c^2q(u,c)x+(u-c)q(u,c)^2}{2q(u,c)^2(x^2-q(u,c))^2}\text{.}$$
Then $f$ equals the above up to a constant, and since adding a constant or multiplying by a nonzero constant amounts to changing $f$ by automorphisms of $\bP^1$, we conclude that we may suppose that $f$ is given by
$$f(x)=\frac{-2u^2x^3-2uq(u,c)x^2-2c^2q(u,c)x+(u-c)q(u,c)^2}{(x^2-q(u,c))^2}\text{.}$$
We now compute the discriminant of the corresponding extension of $\bQ(T)$, that is, the discriminant of
$$-2u^2x^3-2uq(u,c)x^2-2c^2q(u,c)x+(u-c)q(u,c)^2-T(x^2-q(u,c))^2\text{.}$$
One might expect that this discriminant becomes a square over $\bQ(\sqrt{d})$, but surprisingly it actually becomes a square over $\bQ(\sqrt{-d})$. In fact, this discriminant equals
$$-16q(u,c)^3(u-c)^4R(T)^2=-16d^3(u-c)^4R(T)^2$$
where
$$R(T)=(12u^2-8cu+12c^2)T^2+(-15u^3+19cu^2-21c^2u+9c^3)T+(3u^4-2cu^3)\text{.}$$
It follows therefore that, if this discriminant is a square in $\bQ(T)$, then one must be able to take $d=-1$, i.e. $-1$ is represented by the quadratic form $q$, as implied by \eqref{quc}. This is a contradiction, since the quadratic form $q$ is positive definite, and therefore finishes the proof.
\end{proof}

\begin{proof}[Proof of Theorem \ref{nostrong}]
We have previously remarked that a strong alternating line of polynomials over $\bQ$ must give rise to a morphism fitting into either Case 1 or Case 2. Since there are no such morphisms according to Lemma \ref{prenostrong}, the result follows.
\end{proof}

As announced before, we have set the stage for the proof of Theorem \ref{hm}, which we give now.

\begin{proof}[Proof of Theorem \ref{hm}]
We again distinguish between strong alternating lines giving rise to Case 1 and those giving rise to Case 2. As remarked in the proof of Theorem \ref{nostrong}, the argument used to prove that there are no strong alternating lines over $\bQ$ giving rise to Case 1 actually works for proving that there are no strong alternating lines over any number field where $-3$ is not a square. Therefore there are no strong alternating lines giving rise to Case 1 over any quadratic number field apart from possibly $\bQ(\sqrt{-3})$. For this one, strong alternating lines giving rise to Case 1 actually exist; e.g. one may take the line used to prove Theorem \ref{evencase}, which is strong alternating over $\bQ(\sqrt{-3})$ and, coming from Lemma \ref{pqr} for $n=4$, gives rise to Case 1.

We may then focus on lines corresponding to Case 2. Again the argument used in the proof of Lemma \ref{prenostrong} carries through to $\bQ(\sqrt{m})$, except for the part where we used the positive-definiteness of the form $q$. But it is nevertheless true that, if a valid morphism $f$ exists, one must be able to take $d=-1$, so that, according to \eqref{quc}, $-1$ must be represented by the form $q$ over $\bQ(\sqrt{m})$. Conversely, if there exist $u,c\in\bQ(\sqrt{m})$ such that $-1=q(u,c)$, then the line of polynomials defined by
$$P(X)=(X^2-q(u,c))^2$$
and
$$Q(X)=-2u^2X^3-2uq(u,c)X^2-2c^2q(u,c)X+(u-c)q(u,c)^2$$
has Galois group $A_4$ over $\bQ(\sqrt{m})(T)$.

So we are reduced to determining those square-free integers $m$ for which the quadratic form $q$ represents $-1$ over $\bQ(\sqrt{m})$. In order to simplify things we will use classical notation from quadratic form theory, as in \cite{Lam}; in particular, when we talk about the discriminant of a non-degenerate quadratic form over a field $K$, we are regarding the discirminant as an element of $K^{\times}/K^{\times2}$. The form $q$ is equivalent over $\bQ$ to the diagonal form $\langle1,2\rangle$, since $q(x,y)=(x+y)^2+2(x-y)^2$. So according to \cite{Lam}, Corollary I.3.5 we want to determine those square-free values of $m$ for which the form
$$q_0=\langle1,2,1\rangle$$
is isotropic over $\bQ(\sqrt{m})$.

We claim that this happens if and only if the form
$$q_1=\langle1,2,1,2m\rangle$$
is isotropic over $\bQ$. Indeed, if $q_1$ is isotropic over $\bQ$, then $q_0$ represents $-2m$ (\cite{Lam}, Corollary I.3.5) and hence, by \cite{Lam} I.2.3, $q_0\cong\langle a,b,-2m\rangle$ for some $a,b\in\bQ^{\times}/\bQ^{\times2}$. The condition that $q_0$ and $\langle a,b,-2m\rangle$ have the same discriminant forces that $b=-am$, and hence $q_0\cong\langle a,-am,-2m\rangle$. But the latter form is clearly isotropic over $\bQ(\sqrt{m})$, since $a\cdot(\sqrt{m})^2-am\cdot1^2-2m\cdot0^2=0$, and hence $q_0$ is isotropic over $\bQ(\sqrt{m})$. Conversely, suppose that $q_0$ is isotropic over $\bQ(\sqrt{m})$. Since it is clearly not isotropic over $\bQ$, then \cite{Lam}, Theorem VII.3.1 forces that $q_0\cong\langle a,-am,k\rangle$ for some $a$ and some $k$. Now comparing discriminants forces that $2=-km$ in $\bQ^{\times}/\bQ^{\times2}$, or equivalently $k=-2m$. Therefore $q_0$ represents $-2m$, and hence $q_1$ is isotropic over $\bQ$.

Now, in order to check when $q_1$ is isotropic over $\bQ$, we appeal to the Hasse-Minkowski Theorem, which implies that this is equivalent to $q_1$ being isotropic over every completion of $\bQ$. We investigate each completion separately:
\begin{itemize}
\item Over $\bR$: clearly $q_1$ is isotropic if and only if $m<0$.
\item Over $\bQ_p$ for $p\neq 2$: then \cite{Lam}, Corollary VI.2.5(2) implies that already the $3$-dimensional form $\langle1,2,1\rangle$ is isotropic over $\bQ_p$, and therefore so is $q_1$.
\item Over $\bQ_2$: we recall that by \cite{Lam}, Corollary VI.2.15(1) any non-isotropic $4$-dimensional form over a local field has discriminant $1$.  So if $q_1$ is not isotropic over $\bQ_2$ then $4m=1$ in $\bQ_2^\times/\bQ_2^{\times2}$, i.e. $m$ is a square in $\bQ_2$. Since it is a well-known consequence of Hensel's Lemma that squares in $\bQ_2^\times$ are precisely those $2$-adic numbers of the form $4^a(8b+1)$, with $b\in\bZ_2$, and $m$ is assumed to be square-free, it follows that if $q_1$ is non-isotropic over $\bQ_2$ then $m\equiv1\pmod{8}$. Conversely, if $m\equiv1\pmod{8}$, then $m$ is a square in $\bQ_2$ and $q_1\cong\langle1,2,2,1\rangle$, and the latter form is not isotropic over $\bQ_2$.
\end{itemize}
It follows then from the Hasse-Minkowski Theorem that $q_1$ is isotropic over $\bQ$ if and only if $m<0$ and $m\not\equiv1\pmod{8}$, as desired.

\end{proof}

We now know that any alternating line of polynomials of degree $4$ is weak, and we move to a classification of such lines. An interesting feature of the algebra involved is the role played by isomorphisms between $\bP^1$ and non-degenerate conics in $\bP^2$, which are familiar for example from the classification of Pythagorean triples. This is already visible in the following preliminary lemma.

\begin{lemma}
\label{nofixed}
Let $K$ be a number field. Suppose $f:\bP^1_{\overline{\bQ}}\to\bP^1_{\overline{\bQ}}$ is a morphism of degree $4$ defined over $K$ that ramifies above exactly $3$ points with ramification pattern $(2,2)$, and these $3$ points are all defined over $K$. Then the Galois group of the Galois closure of the corresponding extension of $K(T)$ is isomorphic to $V_4$, the Klein four-group.
\end{lemma}
\begin{proof}
Up to automorphisms of $\bP^1_{\overline{\bQ}}$ defined over $K$, we can assume without loss of generality that the points above which $f$ ramifies are precisely $0$, $1$ and $\infty$. The fact that $f$ ramifies above $0$ and $\infty$ with ramification pattern $(2,2)$ means that $f$ can be expressed in homogeneous coordinates as
$$f([x:y])=[ap(x,y)^2:-bq(x,y)^2]$$
where $a,b\in K$ and $p,q\in K[x,y]$ are quadratic forms. The fact that $f$ ramifies above $1$ with ramification pattern $(2,2)$ then implies that $ap(x,y)^2+bq(x,y)^2$ is a square in $\overline{\bQ}[x,y]$, yielding a polynomial identity of the form
\begin{equation}
\label{con}
ap(x,y)^2+bq(x,y)^2+cr(x,y)^2=0
\end{equation}
where $a,b,c\in K$ and $p,q,r\in K[x,y]$ are quadratic forms. If this identity holds, then in particular the conic defined by $aX^2+bY^2+cZ^2=0$ in $\bP^2_K$ has a rational point. If this point has homogeneous coordinates $[\alpha:\beta:\gamma]$, then we may replace $p,q,r$ by $\frac{p}{\alpha},\frac{q}{\beta},\frac{r}{\gamma}$ at the expense of replacing $a,b,c$ by $a\alpha^2,b\beta^2,c\gamma^2$, and now the coefficients sum to $0$. So we may assume without loss of generality that $a+b+c=0$.

Now the equality \eqref{con} implies that
$$\varphi([x:y])=[p(x,y):q(x,y):r(x,y)]$$
defines a degree $2$ morphism from $\bP^1_K$ to the conic $C\subseteq\bP^2_K$ defined by $aX^2+bY^2-(a+b)Z^2=0$. Since this morphism is, up to a linear change of variables, the Veronese embedding of $\bP^1_K$ in $\bP^2_K$ arising from the very ample line bundle $\mathcal{O}_{\bP^1}(2)$, this morphism is automatically an isomorphism. In particular any two such morphisms differ by an automorphism of $\bP^1_K$, and hence the corresponding morphisms $f$ define isomorphic extensions of $K(T)$. So we may without loss of generality work with a specific morphism $\varphi$ as above.

In order to find an appropriate morphism $\varphi$ we use the classical geometric trick of intersecting $C$ with a generic line passing through the rational point $[1:1:1]$. The second intersection of the line through $[1:1:1]$ and $[x:y:0]$ with $C$ turns out to be $[-ax^2+by^2-2bxy:ax^2-by^2-2axy:ax^2+by^2]$, and hence we obtain the morphism
$$\varphi([x:y])=[-ax^2+by^2-2bxy:ax^2-by^2-2axy:ax^2+by^2]\text{.}$$
We obtain, up to a constant, the rational function in affine coordinates
$$f(x)=\frac{(-ax^2-2bx+b)^2}{(ax^2-2ax-b)^2}\text{.}$$
In order to prove that the Galois group of the Galois closure of the corresponding extension of $K(T)$ is $V_4$, we simply compute the cubic resolvent of the polynomial
$$(-aX^2-2bX+b)^2-T((-aX^2-2bX+b)^2-(aX^2-2aX-b)^2)\text{.}$$
This turns out to have cubic resolvent
$$(aX+2b)(aX+2b-4(a+b)T)(a^2X-2ab-4b^2+4b(a+b)T)$$
which factors into linear factors over $K(T)$, and hence the desired Galois group over $K(T)$ is $V_4$, according to Proposition \ref{cubicr}, as desired.
\end{proof}

\begin{proof}[Proof of Theorem \ref{class4}]
Given a line of polynomials of degree $4$ with Galois group $A_4$ over $\bQ(T)$, we assign to it a morphism $f:\bP^1_{\bQ}\to\bP^1_{\bQ}$ in the usual way, and we observe that Lemma \ref{prenostrong} implies that this morphism must fit in the description of Case 3, since Cases 1 and 2 are impossible. It must therefore ramify over exactly $3$ points on the target, and each fiber must have ramification pattern $(2,2)$. Up to changing $f$ by an automorphism of $\bP^1_\bQ$ we may suppose without loss of generality that $f$ maps $\infty$ to $\infty$. Now the absolute Galois group $G_\bQ$ of $\bQ$ acts on the points on the target $\bP^1_\bQ$ over which $f$ ramifies. We will use Lemma \ref{nofixed} to prove that this action does not have any fixed point.

If the action of $G_\bQ$ fixes one of these points, then either it also fixes the other two, in which case all these points are defined over $\bQ$, or the other two are permuted, in which case all three points are defined over a quadratic extension of $\bQ$; the first conclusion implies the second, so we may suppose without loss of generality that all three points are defined over a quadratic extension of $\bQ$. We now look at the cubic resolvent of the polynomial $P-TQ$ in $\bQ(T)$ corresponding to the morphism $f$, and apply Proposition \ref{cubicr}. Since the discriminant of $P-TQ$ is a square, so is the discriminant of the cubic resolvent, and a cubic polynomial with square discriminant is either irreducible or a product of three linear factors. If the second case holds, then the Galois group of $P-TQ$ over $\bQ(T)$ is $V_4$, a contradiction. So the cubic resolvent of $P-TQ$ must be irreducible. But, according to Theorem \ref{cover}, the Galois group of $P-TQ$ over $\bQ(T)$ is $V_4$, since this is the only transitive subgroup of $S_4$ generated by permutations of cycle type $(2,2)$. Therefore the cubic resolvent of $P-TQ$ must factor into linear factors over some extension $K$ of $\bQ$, which we choose minimally. None of these factors lies in $\bQ(T)$, and they are permuted by the obvious action of $G_\bQ$, so this extension must be a cubic extension. We conclude that the Galois group of $P-TQ$ is $V_4$ over $L(T)$, where $L$ is a number field, if and only if $L$ contains the cubic field $K$. On the other hand, since we are assuming that all three points above which $f$ ramifies are defined over a quadratic extension of $\bQ$, Lemma \ref{nofixed} implies that $P-TQ$ has Galois group $V_4$ over that quadratic extension. This is a contradiction.

Therefore the points over which $f$ ramifies form a complete orbit under the action of $G_\bQ$. We can therefore assume that they are the roots of a cubic polynomial, irreducible over $\bQ$, and by translating them appropriately (which amounts to changing our morphism $f$ by an automorphism of $\bP^1_{\bQ}$ on the target, and therefore producing an equivalent line) we can assume that this polynomial is of the form $X^3+kX+m$. Let $\alpha$ be a root of $X^3+kX+m$, and write $f([x:y])=[P(x,y):Q(x,y)]$ in homogeneous coordinates, where as usual $P,Q\in\bQ[x,y]$ are relatively prime homogeneous quartic forms. Then the fact that $f$ ramifies with ramification pattern $(2,2)$ above $\alpha$ means that $P-\alpha Q$ is a square in $\bQ(\alpha)(x,y)$, up to a constant in $\bQ(\alpha)$. In other words we have
\begin{equation}
\label{paq}
P-\alpha Q=u(R+\alpha S+\alpha^2 U)^2
\end{equation}
for some $u\in\mathbb{Q}(\alpha)$ and $R,S,U$ quadratic forms in $\mathbb{Q}[x,y]$. We observe that $u$ is only determined up to multiplication by squares in $\bQ(\alpha)^\times$. Now note that, evaluating \eqref{paq} at $[x:y]=[1:0]$, we have
$$P(1,0)=u(R(1,0)+\alpha S(1,0)+\alpha^2 U(1,0))^2\text{,}$$
where we used that $Q(1,0)=0$, which follows from the assumption that $f$ maps $\infty$ to $\infty$. Since $P(1,0)\in\bQ$ this implies that there is a rational number in the coset of $u$ in $\bQ(\alpha)^\times/\bQ(\alpha)^{\times2}$, and so we may suppose without loss of generality that $u\in\bQ$. Finally, by replacing $P$ and $Q$ by $\frac{P}{u}$ and $\frac{Q}{u}$, respectively, we see that we may suppose without loss of generality that $u=1$. We obtain therefore
\begin{align*}
P-\alpha Q&=(R+\alpha S+\alpha^2 U)^2\\
&=(R^2-2mSU)+\alpha(2RS-2kSU-mU^2)+\alpha^2(S^2+2RU-kU^2)\text{.}
\end{align*}
Since $1,\alpha,\alpha^2$ are linearly independent over $\bQ$, it follows that
\begin{align}
&P=R^2-2mSU\text{,}\label{P}\\
&Q=-2RS+2kSU+mU^2\label{Q}\text{, and}\\
&S^2+2RU-kU^2=0\text{.}\label{R}
\end{align}
We are now in a similar situation to that which we came across in Lemma \ref{nofixed}: by \eqref{R} we see that
$$\varphi([x:y])=[R(x,y):S(x,y):U(x,y)]$$
defines a morphism between $\bP^1_\bQ$ and the conic $C\subseteq\bP^2_\bQ$ given by the homogeneous equation $Y^2+2XZ-kZ^2=0$. Again any such morphism, with $R,S,U$ given by quadratic forms, must be an isomorphism. Moreover, two different such morphisms $\bP^1_\bQ\to C$ differ by an isomorphism of $\bP^1$, and as such they yield equivalent lines (observe that $P$ and $Q$ are determined by $R,S,U$ via \eqref{P} and \eqref{Q}). We may then restrict attention to a specific morphism $\bP^1_\bQ$ to $C$, and to find one we observe that $[1:0:0]$ is a rational point on $C$ and the second intersection between the line through $[1:0:0]$ and $[0:x:y]$ with $C$ is given by
$$[x^2-ky^2:-2xy:-2y^2]\text{.}$$
We may then suppose without loss of generality that $R(x,y)=x^2-ky^2$, $S(x,y)=-2xy$ and $U(x,y)=-2y^2$. Now we use \eqref{P} and \eqref{Q} to conclude that
$$P(x,y)=x^4-2kx^2y^2-8mxy^3+k^2y^4$$
and
$$Q(x,y)=4x^3y+4kxy^3+4my^4\text{.}$$
Upon dividing $Q$ by $4$, which does not change the line up to equivalence, we see that we obtain the morphism of $\bP^1_\bQ$ given in affine coordinates by
$$f(x)=\frac{x^4-2kx^2-8mx+k^2}{x^3+kx+m}\text{.}$$
This corresponds to the line claimed in the statement of Theorem \ref{class4}.

To see that $-4k^3-27m^2$ has to be a square in $\bQ$, we compute the discriminant of
$$X^4-2kX^2-8mX+k^2-T(X^3+kX+m)\text{,}$$
which must be a square in $\bQ(T)$. This equals
$$(-4k^3-27m^2)(T^3+16kT+64m)^2$$
which is indeed a square in $\bQ(T)$ if and only if $-4k^3-27m^2$ is a square in $\bQ$, as desired.
\end{proof}

\providecommand{\bysame}{\leavevmode\hbox to3em{\hrulefill}\thinspace}

\end{document}